\documentclass[11pt,runningheads]{llncs}
\usepackage{amsmath,amssymb}
\usepackage{mathrsfs}
\usepackage{xspace}
\usepackage{esvect}
\usepackage{tikz-cd}
\usepackage{xcolor}
\usepackage{url}

\newtheorem{thm}{Theorem}

\newtheorem{prop}[thm]{Proposition}

\newcommand{\ie}{\textit{i.e.}\xspace}

\newcommand{\Z}{\ensuremath{\mathbb{Z}}\xspace}
\newcommand{\R}{\ensuremath{\mathbb{R}}\xspace}

\newcommand{\Sig}{\mathfrak{S}}

\renewcommand{\leq}{\leqslant}
\renewcommand{\geq}{\geqslant}

\renewcommand{\epsilon}{\varepsilon}
\renewcommand{\phi}{\varphi}
\renewcommand{\rho}{\varrho}

\newcommand{\Restriction}{\mathpunct{\mathpunct{\restriction}}}

\newcommand{\eval}[2][\right]{\relax\ifx#1\right\relax \left.\fi#2#1\rvert}

\newcommand{\Inf}{\operatorname{\bf Inf}}
\newcommand{\AS}{\operatorname{\bf AS}}
\newcommand{\I}{\AS}
\newcommand{\E}{\operatorname{\bf E}}
\newcommand{\Var}{\operatorname{\bf Var}}
\newcommand{\Prob}{\operatorname{\bf Prob}}

\begin{document}
\title{Average sensitivity of nested canalizing multivalued functions}
\author{Élisabeth Remy\inst{1} \and Paul Ruet\inst{2}}
\institute{Aix Marseille Univ, CNRS, I2M, Marseille, France \\
\email{elisabeth.remy@univ-amu.fr} 
\and
CNRS, Université Paris Cité, Paris, France \\
\email{ruet@irif.fr} 
}
\maketitle

\begin{abstract}
The canalizing properties of biological functions have been mainly studied in the context of Boolean modelling of gene regulatory networks. An important mathematical consequence of canalization is a low average sensitivity, which ensures in particular the expected robustness to noise. In certain situations, the Boolean description is too crude, and it may be necessary to consider functions involving more than two levels of expression. We investigate here the properties of nested canalization for these multivalued functions. We prove that the average sensitivity of nested canalizing multivalued functions is bounded above by a constant. In doing so, we introduce a generalization of nested canalizing multivalued functions, which we call weakly nested canalizing, for which this upper bound holds.
\end{abstract}

\section{Introduction}


The concept of canalization in biology was proposed by Waddington in the early 1940s \cite{Wad42}. It corresponds to the property of a biological process of being able to produce a relatively stable phenotype despite the presence of variability, as a kind of noise filter inherent in the process \cite{Pod21,DLR19}. This phenomenon is observed in natural systems, for example in \cite{Sur08} it has been shown that, at the molecular level, the development process of the Drosophilia embryo is canalized: the expression of genes controlling embryo segmentation is concentrated in a small area of state space.

Canalizing Boolean functions form a class of Boolean functions introduced by Kauffman \cite{Kau93,Kau03} that formalize this canalizing behaviour observed in gene regulatory networks. 
In short, letting $\Z/{2\Z}$ denote the $2$-element ring of integers modulo $2$, \emph{canalizing} Boolean functions are functions $f$ from $(\Z/{2\Z})^n$ to $\Z/{2\Z}$ (or possibly to $\R$) such that at least one input variable, say $x_i$ ($1\leq i\leq n$), has a value $a=0$ or $1$ which determines the value of $f(x)$. \emph{Nested canalizing} (NC) functions provide a ``recursive'' version of canalizing functions: an NC function $f$ is canalizing and, moreover, its restriction $f\Restriction_{x_i\neq a}$ is itself NC.


It is worth noting that NC Boolean functions are both rare among the set of Boolean functions and frequent among Boolean functions modeling gene networks. Indeed, on the one hand, they form a sparse set of Boolean functions \cite{Jus04}: the fraction of NC functions of arity $n$ among all Boolean functions of arity $n$ tends to $0$ as $n\rightarrow\infty$. On the other hand, \cite{Sub22} gives evidence that gene regulatory networks, which are built from biological data and knowledge from the literature \cite{Tho73}, are far from random: in particular, NC functions are predominant in Boolean gene networks. This canalizing property can be, and has been \cite{HJ12,Zho16}, used as a guide to filter through the large number of candidate functions to parameterise the regulatory graph, a critical point in the process of modelling networks.

On the mathematical side, a striking feature of NC functions is their low \emph{average sensitivity}. The average sensitivity $\AS(f)$ of a Boolean function $f$ is a measure of its ``complexity'' in the sense of Boolean functions analysis \cite{ODon14}: roughly speaking, for $f:(\Z/{2\Z})^n\rightarrow\Z/{2\Z}$, $\AS(f)$ measures how scattered the frontier between $0$'s and $1$'s is. For arbitrary Boolean functions, $\AS(f)=\mathcal{O}(n)$, but some functions have significantly lower average sensitivity. For NC functions, $\AS(f)$ is bounded above by a constant \cite{Lau13,Klo13}. This low average sensitivity has several consequences. Most importantly, it entails noise stability (noise in inputs is not amplified \cite{Sch08,ODon14}), a robustness property observed indeed in gene networks \cite{Kau93,Kau04}. Functions with low $\AS$ also depend on few coordinates \cite{Fri98}. And more theoretically, the Fourier-Walsh spectrum of functions with low $\AS$ is concentrated on low degrees, and the function can be more easily learned from examples \cite{ODon14}.

If in most cases, Boolean variables are sufficient to capture the main and key characteristics of the system under study, in some situations this description is too crude, and it may be necessary to consider more levels. To model such a situation correctly, multivalued variables have been introduced \cite{Tho91}. Then it is necessary to consider multivalued functions $f:(\Z/{k\Z})^n\rightarrow\Z/{k\Z}$ or $\R$ for some $k\geq 2$, where $\Z/{k\Z}$ denotes the ring of integers modulo $k$. The notion of average sensitivity generalizes to the multivalued setting \cite{ODon14}, and multivalued NC functions are defined in \cite{ML11,ML12}. Very little is known about their spectral properties. In \cite{KLKAL17}, a variant of average sensitivity, the normalized average $c$-sensitivity, is defined for multivalued functions, and used to measure the stability of networks based on NC functions.

A natural question is whether the average sensitivity of NC multivalued functions is bounded above by a constant, too. We prove in Theorem \ref{mainth} that this is the case. We actually show that the upper bound holds for a more general class of functions, which we call \emph{weakly nested canalizing} (WNC), and at the same time this enables us to establish the upper bound in a simpler way than in \cite{Lau13} for NC Boolean functions.

The paper is organized as follows. In Section \ref{secncf}, we recall the definition of nested canalizing multivalued functions from \cite{ML11,ML12} and illustrate it with example functions inspired from the logical modelling of the phage lambda. In Section \ref{sec:wnc}, we define weakly nested canalizing multivalued functions, provide examples of multivalued functions which are WNC but not NC, including a function arising from the phage lambda modelling. We also give an alternative characterization of WNC functions and prove that NC functions are indeed WNC. In Sections \ref{as} and \ref{upper}, we recall the definition of average sensitivity and prove Theorem \ref{mainth}. Section \ref{concl} concludes with possible perspectives for further research.

\section{Nested canalizing multivalued functions}

\label{secncf}

\subsection{Definition}
Let $k,n$ be positive integers, $k\geq 2$. $\Z/{k\Z}$ is the ring of integers modulo $k$.


Following \cite{ML11,ML12,KLKAL17}, we shall say that $f:(\Z/{k\Z})^n\rightarrow\Z/{k\Z}$ is \emph{canalizing with respect to coordinate $i$ and $(a,b)\in\Z/{k\Z}\times\Z/{k\Z}$} if there exists a function $g:(\Z/{k\Z})^n\rightarrow\Z/{k\Z}$ different from the constant $b$ such that
$$
f(x)=
\begin{cases}
b & \text{if } x_i=a \\
g(x) & \text{if } x_i\neq a.
\end{cases}
$$
We shall simply say that $f$ is \emph{canalizing} if it is canalizing with respect to some $i,a,b$.

A \emph{segment} is a (proper, nonempty) subset of $\Z/{k\Z}$ of the form $\{0,\ldots,i\}$ or $\{i,\ldots,k-1\}$, with $0\leq i\leq k-1$.

Let $\sigma\in\Sig_n$ be a permutation, $A_1,\ldots,A_n$ be segments, and $c_1,\ldots,$ $c_{n+1}\in\Z/{k\Z}$ be such that $c_n\neq c_{n+1}$. Then $f$ is said to be \emph{nested canalizing (NC) with respect to $\sigma$, $A_1,\ldots,A_n$, $c_1,\ldots,c_{n+1}$} if 
$$
f(x)=
\begin{cases}
c_1 & \text{if } x_{\sigma(1)}\in A_1 \\
c_2 & \text{if } x_{\sigma(1)}\notin A_1,x_{\sigma(2)}\in A_2 \\
\;\vdots & \quad\vdots \\
c_n & \text{if } x_{\sigma(1)}\notin A_1,\ldots,x_{\sigma(n-1)}\notin A_{n-1},x_{\sigma(n)}\in A_n \\
c_{n+1} & \text{if } x_{\sigma(1)}\notin A_1,\ldots,x_{\sigma(n-1)}\notin A_{n-1},x_{\sigma(n)}\notin A_n.
\end{cases}
$$
We shall simply say that $f$ is \emph{NC} if it is NC with respect to some $\sigma$, $A_1,\ldots,A_n$, $c_1,\ldots,c_{n+1}$.

\subsection{Example: logical modelling of the phage lambda}

\label{sec:phage}

The following example is inspired from the logical modelling of the phage lambda, a model system whose study revealed the basic concepts and mechanistic details of gene regulation. This regulator model that has been widely studied to understand the decision between lysis and lysogenization \cite{Pta92,RR08,Dan21} is described by NC functions.

The model involves two genes, CI and Cro. CI is either expressed or not, and its expression level is therefore modelled by a Boolean variable, while Cro can take 3 values $\{0,1,2\}$ \cite{TT95}. This simple model is sufficient to display both multistability (representing lysis and lysogeny fates) and oscillations (lysogeny state) \cite{RR08b,Rue16}.

In state $x=(x_{CI},x_{Cro})\in\Z/{2\Z}\times\Z/{3\Z}$, the next target value of CI is given by the following function $f_{CI}:\Z/{2\Z}\times\Z/{3\Z}\rightarrow\Z/{2\Z}$:
$$
f_{CI}(x) =
\begin{cases}
0 & \text{if } x_{Cro}\geq 1 \\
1 & \text{otherwise.}
\end{cases}
$$ 
For instance, in state $(1,2)$, the next value of CI can be $0$ because $f_{CI}(1,2)$ $=0$, and in state $(0,2)$, the value of CI cannot change because $f_{CI}(0,2)=0$.
Similarly, the target value of Cro is given by a function $f_{Cro}:\Z/{2\Z}\times\Z/{3\Z}\rightarrow\Z/{3\Z}$. 
For instance,
$$
f_{Cro}(x) =\begin{cases}
0 & \text{if } x_{CI}=1 \\
1 & \text{if } x_{CI}=0 \text{ and } x_{Cro}=2 \\
2 & \text{otherwise.}
\end{cases}
\qquad
\begin{array}{c|rc}
2&\;1&0 \\
1&2&0 \\
0&2&0 \\ \hline
\overset{\scriptstyle x_{Cro}\kern0.5em\mathstrut}%
    {\quad\scriptstyle x_{CI}\;\;\;\kern-0.5em\mathstrut}
&0&1
\end{array}
$$
In the above table to the right, the $x=(x_{CI},x_{Cro})$ entry is the value of $f_{Cro}(x)$. In state $(1,2)$, the target value of Cro is $0$ because $f_{Cro}(1,2)=0$, so the value of Cro can decrease.

In this context of discrete dynamics, to represent the trajectories of the dynamics of the whole system, we need to choose the update rule for the pair of functions $(f_{CI}, f_{Cro})$. We choose the \emph{asynchronous} setting, which means that at each time step, the level of at most one gene can change. So, in state $(1,2)$, since the levels of both genes can decrease ($f_{CI}(1,2)=0<1=x_{CI}$ and $f_{Cro}(1,2)=0<2=x_{Cro}$), the system can (non-deterministically) reach either state $(1,1)$ or state $(0,2)$. Note that there is no direct transition from $(1,2)$ to $(1,0)$ because we limit the length step to 1.
The set of asynchronous trajectories is summarized in the following state transition graph, where the vertices are the system states $(x_{CI},x_{Cro})$ and the arrows connect two consecutive states:

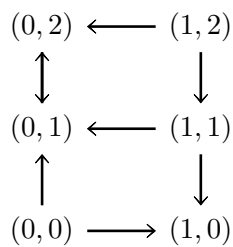
\begin{figure}
\normalsize
\centering
\tikzcdset{arrow style=tikz, arrows={line width=0.35mm},
    diagrams={>={Stealth[round,length=5pt,width=8pt,inset=4pt]}}}
\begin{tikzcd}
(0,2) \arrow{d}{} & (1,2) \arrow{d}{} \arrow{l}{} \\
(0,1) \arrow{u}{} & (1,1) \arrow{d}{} \arrow{l}{} \\
(0,0) \arrow{u}{} \arrow{r}{} & (1,0) 
\end{tikzcd}
\caption{\label{stg} State transition graph : asynchronous trajectories of the system $(f_{CI}, f_{Cro})$. It contains two attractors: the stable state $(1,0)$ and the cyclical attractor $\{ (0,1), (0,2)\}$ }
\end{figure}

The presence of two attractors (terminal strongly connected components, see Figure \ref{stg}) reflects multistability: the stable state $(1,0)$ representing lysis fate and the cyclical attractor lysogeny fate.

Clearly, $f_{CI}$ and $f_{Cro}$ are both NC. 
To see from the above definition of $f_{CI}$ that it is NC, it suffices to let the segment for $x_{Cro}$ be $\{1,2\}\subset\{0,1,2\}$.
It is again easy to see that $f_{Cro}$ is NC: the segment corresponding to $x_{CI}$ is $\{1\}\subset\{0,1\}$, and then the segment corresponding to $x_{Cro}$ is $\{2\}\subset\{0,1,2\}$.

\section{Weakly nested canalizing multivalued functions}

\label{sec:wnc}

In Theorem \ref{mainth}, we shall give an upper bound on average sensitivity which holds not only for NC functions, but for the more general class of weakly nested canalizing functions, which we define now.

\subsection{Definition}

Let $n$ be a positive integer. For each $i\in\{1,\ldots,n\}$, $\Omega_i$ is a finite set of cardinality $k_i>0$, $\Omega=\prod_i\Omega_i$, and $f:\Omega\rightarrow\R$.
Note that we do not require $k_i\geq 2$ for all $i$. If $k_j=1$ for some $j$, $f$ could be viewed as a function with one less variable, \ie as a function on $\prod_{i\neq j}\Omega_i$, but we still consider it as a function defined on $\prod_i\Omega_i$.

We shall say that $f$ is \emph{weakly canalizing with respect to coordinate $i$ and $(a,b)\in\Omega_i\times\R$} 
if $f(x)=b$ whenever $x_i=a$, and simply that it is \emph{weakly canalizing} if it is weakly canalizing with respect to some $i,a,b$.

Note that this definition differs slightly from the usual definition by the absence of condition on the values of $f$ for $x_i\neq a$: we do not require the existence of some $x$ such that $x_i\neq a$ and $f(x)\neq b$. In particular, constant functions are weakly canalizing, though not canalizing.

If $f$ is canalizing with respect to $i,a,b$ and $k_i\geq 2$, we shall consider
$$
f\Restriction_{x_i\neq a}:\Omega\cap\{x\mid x_i\neq a\}\rightarrow\R,
$$
the restriction of $f$ to the set of $x\in\Omega$ such that $x_i\neq a$.

The class of weakly nested canalizing functions on $\Omega=\prod_i\Omega_i$ is then defined by induction on the cardinality $|\Omega|=\prod_ik_i$ of $\Omega$:
\begin{itemize}
\item If $|\Omega|=1$, \ie $k_i=1$ for all $i$, any $f:\Omega\rightarrow\R$ is \emph{weakly nested canalizing (WNC) on $\Omega$}.
\item If $|\Omega|>1$, $f:\Omega\rightarrow\R$ is \emph{WNC on $\Omega$} if it is weakly canalizing with respect to some $i,a,b$ such that $k_i\geq 2$ and $f\Restriction_{x_i\neq a}$ is WNC on $\Omega\cap\{x\mid x_i\neq a\}$, a strict subset of $\Omega$.
\end{itemize}

\subsection{Examples of WNC non NC functions}

As we shall see in Proposition \ref{ncwnc}, the class of WNC functions contains the class of NC functions. We give here a few examples of simple multivalued functions which are WNC but not NC.

\begin{itemize}
\item As we have already observed, constant functions from $(\Z/{k\Z})^n$ to $\Z/{k\Z}$ are WNC but not NC.
\item In decomposing a WNC function $f:(\Z/{k\Z})^n\rightarrow\Z/{k\Z}$, it is possible to ``peel'' a coordinate hyperplane defined on some coordinate $i$ (\ie by some equation $x_i=a$), then a coordinate hyperplane defined on $j$, and later a coordinate hyperplane defined on $i$ again. This is because of the recursive nature of the definition of WNC functions, and gives more freedom in the construction of WNC functions than in the construction of NC functions.

For instance, the functions $\min$ and $\max:(\Z/{k\Z})^2\rightarrow\Z/{k\Z}$ are not NC, as observed in \cite{KLKAL17}. However, an easy induction on $k$ shows that they are WNC. For instance, $\min=\min_k:\{0,\ldots,k-1\}^2\rightarrow\{0,\ldots,k-1\}$ is weakly canalizing with respect to $1,0,0$, $\min_k\Restriction_{x_1\neq 0}$ is weakly canalizing with respect to $2,0,0$, and $\min_k\Restriction_{x_1\neq 0,x_2\neq 0}$ is identical to the function $\min_{k-1}:\{1,\ldots,k-1\}^2\rightarrow\{1,\ldots,k-1\}$, which is WNC.
\item Also, in constructing a WNC function $f:(\Z/{k\Z})^n\rightarrow\Z/{k\Z}$, the values $a$ used to define $f(x)$ for $x_i=a$ need not be extremal values (initially $0$ or $k-1$), they can be intermediate values: $0<a<k-1$.

For instance, the function from $\Z/{3\Z}$ to $\Z/{3\Z}$ defined by $0\mapsto 0, 1\mapsto 1, 2\mapsto 0$ is not NC because it is canalizing with respect to either the intermediate value $1$ (for its unique variable), or the values $0$ and $2$ (which do not form a segment). But any function from $\Z/{k\Z}$ to $\Z/{k\Z}$ is WNC.
\end{itemize}

\subsection{Back to the phage lambda example }

\label{sec:phage2}

Getting back to the model of the phage lambda presented in Section \ref{sec:phage}, since $f_{CI}$ and $f_{Cro}$ are NC, they are also WNC by Proposition \ref{ncwnc} below.

Let us consider the following function:
$$
f'_{Cro}(x) =
\begin{cases}
1 & \text{if } x_{Cro}=2 \\
0 & \text{if } x_{Cro}\neq 2 \text{ and } x_{CI}=1 \\
2 & \text{if } x_{Cro}\neq 2, x_{CI}\neq 1 \text{ and } x_{Cro}=1 \\
1 & \text{otherwise}
\end{cases}
\hspace*{8mm}
\begin{array}{c|rc}
2&\;1&1 \\
1&2&0 \\
0&1&0 \\ \hline
\overset{\scriptstyle x_{Cro}\kern0.5em\mathstrut}%
    {\quad\scriptstyle x_{CI}\;\;\;\kern-0.5em\mathstrut}
&0&1
\end{array}
$$
In the above table to the right, the $x=(x_{CI},x_{Cro})$ entry is the value of $f'_{Cro}(x)$. By the same argument as developed in Section \ref{sec:phage}, we can see that $(f_{CI}, f'_{Cro})$ gives rise to the same state transition graph as $(f_{CI}, f_{Cro})$ (represented in Figure \ref{stg}).  

It is interesting to remark that $f_{Cro}$ and $f'_{Cro}$ do not have the same canalizing property: $f_{Cro}$ is NC, while $f'_{Cro}$ is clearly WNC but not NC: indeed, the table shows that the first canalizing step has to be $x_{Cro}=2$ (unique choice which fixes the value of $f'_{Cro}$), the resulting restriction $f'_{Cro}\restriction_{x_{Cro}\neq 2}$ is canalizing only with $x_{CI}=1$, and the restriction $f'_{Cro}\restriction_{x_{Cro}\neq 2,x_{CI}\neq 1}$ is not constant. Thus, in this example two functions that represent the same asynchronous dynamics do not have the same canalizing properties.

\subsection{Properties of WNC functions}

Intuitively, a function $f:\Omega\rightarrow\R$ is WNC if its domain $\Omega$ can be ``peeled'' by successively removing coordinate hyperplanes (defined by equations of the form $x_i=a$) whose points are mapped by $f$ to the same value, whence the following characterization:

\begin{prop}[Characterization]
\label{prop:charac}
Letting $K=\sum_ik_i$, $f$ is WNC if and only if there exist a function $v:\{1,\ldots,K\}\rightarrow\{1,\ldots,n\}$ and numbers $a_i\in\Omega_{v(i)}$ and $b_i\in\R$ for each $i\in\{1,\ldots,K\}$ such that:
$$
f(x)=
\begin{cases}
b_1 & \text{if } x_{v(1)}=a_1 \\
b_2 & \text{if } x_{v(1)}\neq a_1,x_{v(2)}=a_2 \\
\;\vdots & \quad\vdots \\
b_K & \text{if } x_{v(1)}\neq a_1,\ldots,x_{v(K-1)}\neq a_{K-1},x_{v(K)}=a_K. \\
\end{cases}
$$
\end{prop}

\begin{proof}
By induction on $K$.

The base case corresponds to $|\Omega|=1$. Then for all $i$, $\Omega_i$ is a singleton $\{a_i\}$, and $K=n$. In that case, a function $f:\Omega\rightarrow\R$ is simply a number $b\in\R$ and can be defined for instance by
$$
f(x)=
\begin{cases}
b & \text{if } x_{1}=a_1 \\
b_2 & \text{if } x_{1}\neq a_1,x_{2}=a_2 \\
\;\vdots & \quad\vdots \\
b_n & \text{if } x_{1}\neq a_1,\ldots,x_{n-1}\neq a_{n-1},x_{n}=a_n, \\
\end{cases}
$$
where $b_2,\ldots,b_n$ are arbitrary numbers because the conditions of lines $2$ to $n$ are obviously not matched.

We now assume that $|\Omega|>1$ and that $f:\Omega\rightarrow\R$ is WNC on $\Omega$: this means that $f$ is weakly canalizing with respect to some $i,a,b$ such that $k_i\geq 2$ and that $f'=f\Restriction_{x_i\neq a}$ is WNC on $\Omega'=\Omega\cap\{x\mid x_i\neq a\}$. Therefore
$$
f(x)=
\begin{cases}
b & \text{if } x_{i}=a \\
f'(x) & \text{if } x_i\neq a.
\end{cases}
$$
Moreover, for $f':\Omega'\rightarrow\R$, we have $K'=K-1$, hence by the induction hypothesis, there exist a function $v:\{2,\ldots,K\}\rightarrow\{1,\ldots,n\}$ and numbers $a_i\in\Omega_{v(i)}$ and $b_i\in\R$ indexed by $i\in\{2,\ldots,K\}$ (a set of cardinality $K'$) such that
$$
f'(x)=
\begin{cases}
b_2 & \text{if } x_{v(2)}=a_2 \\
b_3 & \text{if } x_{v(2)}\neq a_2,x_{v(3)}=a_3 \\
\;\vdots & \quad\vdots \\
b_{K} & \text{if } x_{v(2)}\neq a_2,\ldots,x_{v(K-1)}\neq a_{K-1},x_{v(K)}=a_{K}. \\
\end{cases}
$$
This entails the following expression for $f$:
$$
f(x)=
\begin{cases}
b & \text{if } x_{i}=a \\
b_2 & \text{if } x_i\neq a,x_{v(2)}=a_2 \\
b_3 & \text{if } x_i\neq a,x_{v(2)}\neq a_2,x_{v(3)}=a_3 \\
\;\vdots & \quad\vdots \\
b_{K} & \text{if } x_i\neq a,x_{v(2)}\neq a_2,\ldots,x_{v(K-1)}\neq a_{K-1},x_{v(K)}=a_{K},
\\
\end{cases}
$$
which is of the expected form by letting $a_1=a,b_1=b$ and extending $v$ to $v:\{1,2,\ldots,K\}\rightarrow\{1,\ldots,n\}$ with $v(1)=i$.
\qed
\end{proof}

In decomposing an NC function $f:(\Z/{k\Z})^n\rightarrow\Z/{k\Z}$, each coordinate $i\in\{1,\ldots,n\}$ is considered exactly once (in some order prescribed by a permutation $\sigma$) and the value of $f$ is fixed for $x_{\sigma(i)}$ in some segment $A_i$. This can be realized by successively fixing the value of $f$ for each $\alpha\in A_i$, and therefore, the class of WNC functions contains the class of NC functions, as stated in the following Proposition:

\begin{prop}[NC $\Rightarrow$ WNC]
\label{ncwnc}
If $f:(\Z/{k\Z})^n\rightarrow\Z/{k\Z}$ is NC, then it is WNC.
\end{prop}

\begin{proof}
Assume $f$ is NC with respect to $\sigma$, $A_1,\ldots,A_n$, $c_1,\ldots,c_{n+1}$. For each $i\in\{1,\ldots,n\}$, let
\begin{align*}
A_i &= \{\alpha_i^1,\ldots,\alpha_i^{|A_i|}\} \\
(\Z/{k\Z})\setminus A_i &= \{\alpha_i^{1+|A_i|},\ldots,\alpha_i^k\}
\end{align*}
with $\alpha_i^1<\cdots<\alpha_i^{|A_i|}$ and $\alpha_i^{1+|A_i|}<\cdots<\alpha_i^k$. This defines $K=nk$ numbers $\alpha_i^j\in\Z/{k\Z}$. For each $i\in\{1,\ldots,n\}$ and $j\in\{1,\ldots,k\}$, let
$$
\beta_i^j=
\begin{cases}
c_i & \text{if } j\leq |A_i| \\
c_{n+1} & \text{otherwise.}
\end{cases}
$$
To comply with the characterization of WNC functions (Proposition \ref{prop:charac}), we relabel the numbers $\alpha_i^j,\beta_i^j$ by identifying the list
$$
\alpha_1^1,\ldots,\alpha_1^{|A_1|},\ldots,\alpha_n^1,\ldots,\alpha_n^{|A_n|},\alpha_1^{1+|A_1|},\ldots,\alpha_1^k,\ldots,\alpha_n^{1+|A_n|},\ldots,\alpha_n^k
$$
as the list $a_1,\ldots,a_K$, and by identifying similarly the list
$$
\beta_1^1,\ldots,\beta_1^{|A_1|},\ldots,\beta_n^1,\ldots,\beta_n^{|A_n|},\beta_1^{1+|A_1|},\ldots,\beta_1^k,\ldots,\beta_n^{1+|A_n|},\ldots,\beta_n^k
$$
as the list $b_1,\ldots,b_K$. Call $\phi$ this relabelling, which maps $r\in\{1,\ldots,K\}$ to the pair $\phi(r)=(i,j)$ such that $a_r=\alpha_i^j$ and $b_r=\beta_i^j$. For instance, $\phi(1)=(1,1)$ and $\phi(K)=(n,k)$. Then finally, a function $v:\{1,\ldots,K\}\rightarrow\{1,\ldots,n\}$ is defined by $v(r)=\sigma(i)$ if $\phi(r)=(i,j)$. Then $f$ clearly enjoys
the characterization of WNC functions, with the choice of function $v$ and numbers $a_r, b_r$.
\qed
\end{proof}

\section{Average sensitivity}

\label{as}

In Section \ref{upper} we shall be interested in the average sensitivity of (WNC) multivalued functions, but before giving the (more technical) definition of average sensitivity for multivalued functions, we start by recalling the more intuitive definition for Boolean functions.

\subsection{Boolean-valued Boolean functions}

The average sensitivity of a Boolean function $f:(\Z/2\Z)^n\rightarrow\Z/2\Z$ is the probability that the $i$th variable affects the outcome. It can be defined as follows.
For $x\in(\Z/2\Z)^n$, let $\overline{x}^i$ denote the vector obtained from $x$ by changing the value of $x_i$. Then the \emph{influence of the $i$th variable} is
$$
\Inf_i[f]=\Prob_x[f(x) \neq f(\overline{x}^i)]\in[0,1],
$$
where the probability is taken for the uniform distribution, and the \emph{average sensitivity} (also called \emph{influence} or \emph{total influence}) is
$$
\I[f]=\sum_i\Inf_i[f]\in[0,n].
$$

This can be reformulated in terms of boundary edges. Let an \emph{edge} in the Hamming cube $(\Z/2\Z)^n$ be a pair $(x,y)$ such that the Hamming distance between $x$ and $y$ is $1$, and let a \emph{boundary edge} be an edge $(x,y)$ such that $f(x)\neq f(y)$. Then the average sensitivity $\I[f]$ is such that the fraction of edges which are boundary edges equals $\I[f]/n$.

\subsection{Multivalued functions}

This definition can be generalized to multivalued functions. Following \cite[Chapter~8]{ODon14}, we shall take the following definition.

First, Fourier decomposition is generalized to non Boolean domains. Let $\Omega=\prod_{i=1}^n\Omega_i$ be as above, with $|\Omega_i|=k_i$. On the vector space of real-valued functions defined on $\Omega$, an inner product is given by $\langle f,g\rangle=\E_x[f(x)g(x)]$, where $\E$ denotes the expectation. Here, $x\in\Omega$ and we assume independent uniform probability distributions on the $\Omega_i$. A \emph{Fourier basis} is an orthonormal basis $(\phi_\alpha)_{\alpha\in\prod_i\{0,\ldots,k_i\}}$ such that $\phi_{(0,\ldots,0)}=1$. It is not difficult to see that a Fourier basis always exists, although it is not unique.

Then, fix a Fourier basis $(\phi_\alpha)$. The \emph{Fourier coefficients of $f:\Omega\rightarrow\R$} are $\widehat{f}(\alpha)=\langle f,\phi_\alpha\rangle$, and $E_if=\sum_{\alpha|\alpha_i=0}\widehat{f}(\alpha)\phi_\alpha$ turns out to be independent of the basis. The notation $\sum_{\alpha|\alpha_i=0}$ denotes the sum for all $\alpha$ such that $\alpha_i=0$. For all $i\in\{1,\ldots,n\}$, let the \emph{$i$th coordinate Laplacian operator} $L_i$ be the linear operator defined by $L_if=f-E_if$.

Finally, the \emph{influence of coordinate $i$ on $f$} is defined by $\Inf_i[f]=\langle f,L_if\rangle$, and the \emph{average sensitivity} of $f$ is then $\I[f]=\sum_i\Inf_i[f]$.

By Plancherel’s theorem (see \cite{ODon14}), we have
$$
\Inf_i[f]=\sum_{\alpha_i\neq 0}\widehat{f}(\alpha)^2=\E_x[\Var_{y_i}[f(x_1,\ldots,x_{i-1},y_i,x_{i+1},\ldots,x_n)]],
$$
where $\Var$ denotes the variance ($\Var[g]=\E[g^2]-\E[g]^2$) and $y_i\in\Omega_i$. The above equality makes clear that the definition of influence of multivalued functions generalizes the Boolean case.


\section{Upper bound on the average sensitivity of WNC multivalued functions}

\label{upper}

For an arbitrary $f:\Omega\rightarrow[0,M]$, we have $\Var_i[f](x)\leq (M/2)^2$ for all $i$, therefore $\Inf_i[f]\leq M^2/4$ for all $i$ and
$$
\I[f]\leq n\cdot M^2/4=\mathcal{O}(n).
$$
For WNC functions, this upper bound can be greatly improved. In the Boolean case, \cite{Lau13} proves (by a different method from ours) that $\I[f]\leq 2$ for NC $\{-1,+1\}$-valued functions. This bound is improved in \cite{Klo13}, where it is proved that $\I[f]\leq 4/3$. For NC functions $f:\{0,1\}^n\rightarrow\{0,1\}$, the result in \cite{Lau13} means $\I[f]\leq 1/2$.

Theorem \ref{mainth} generalizes this result, by establishing that, in the more general multivalued case, the average sensitivity of WNC functions is bounded above by a constant.

\begin{thm}
\label{mainth}
Let $\Omega=\prod_{i=1}^n\Omega_i$ where each $\Omega_i$ has cardinality $k_i>0$. Let $f:\Omega\rightarrow[0,M]$ and $\kappa=\max_i(k_i-1)/k_i<1$. If $f$ is WNC (in particular if it is NC), then
$$
\I[f]\leq\frac{M^2}{4(1-\kappa)}.
$$
\end{thm}

\begin{proof}
We prove this by induction on $\sum_ik_i$. If $\sum_ik_i=n$, \ie $k_i=1$ for all $i$, the inequality holds trivially: actually $\I[f]=0$.

Now assume $\sum_ik_i>n$ and $f$ is WNC. This means that $f$ is weakly canalizing with respect to some $j,a,b$ such that $k_j\geq 2$, and we let $f'=f\Restriction_{x_j\neq a}$. Let $\Omega'$ be the set of $x\in\Omega$ such that $x_j\neq a$, so that $f':\Omega'\rightarrow[0,M]$. The induction hypothesis applied to $f'$ reads
$$
\I[f']\leq\frac{M'^2}{4(1-\kappa')}
$$
with
\begin{align*}
M' & = \max_{x\in\Omega'}f'(x) = \max_{x\in\Omega'}f(x) \leq M \\
\kappa' & = \max\left\{\frac{k_j-2}{k_j-1}, \max_{i\neq j}\frac{k_i-1}{k_i}\right\} \leq \kappa.
\end{align*}
Note that the induction hypothesis implies $\I[f']\leq M^2/(4(1-\kappa))$. We shall use the notation
$$
\Var_i[f](x)=\Var_{y_i}[f(x_1,\ldots,x_{i-1},y_i,x_{i+1},\ldots,x_n)].
$$
Then $\I[f]=\E_x[\sum_i\Var_i[f](x)]$ and
\begin{align*}
\I[f] \cdot \prod_ik_i & = \sum_x\sum_i \Var_i[f](x) \\
& = \sum_{x_j=a} \Var_j[f](x) +
\sum_{x_j\neq a} \biggl(\Var_j[f](x) + \sum_{i\neq j} \Var_i[f](x)\biggr)
\end{align*}
since $f(x)$ is constant when $x_j=a$, so that $\Var_i[f](x)=0$ for $i\neq j$. Here and below, the notations $\sum_{x_j=a}$ and $\sum_{x_j\neq a}$ denote the sums for all $x$ such that $x_j=a$ (resp. $x_j\neq a$). Fur\-ther\-more, $\Var_j[f](x)$ is independent of $x_j$, and on the other hand, $\Var_i[f](x)$ $=\Var_i[f'](x)$ when $x_j\neq a$ and $i\neq j$. Thus
\begin{equation*}
\I[f] \cdot \prod_ik_i = k_j \cdot \sum_{x_j=a} \Var_j[f](x) +
\sum_{x_j\neq a} \sum_{i\neq j} \Var_i[f'](x).
\end{equation*}
Since $0\leq f(x)\leq M$ for all $x$, we have $\Var_j[f](x)\leq M^2/4$, and on the other hand, $x_j\neq a \Leftrightarrow x\in\Omega'$. Therefore
\begin{align*}
\I[f] \cdot \prod_ik_i & \leq k_j \cdot \prod_{i\neq j}k_i \cdot M^2/4 +
\sum_{x\in\Omega'} \sum_{i=1}^n \Var_i[f'](x) \\
& = \prod_ik_i \cdot M^2/4 + \I[f'] \cdot (k_j-1) \cdot \prod_{i\neq j}k_i
\end{align*}
and 
\begin{equation*}
\I[f] \leq \frac{M^2}{4} + \I[f'] \cdot \frac{k_j-1}{k_j} \leq \frac{M^2}{4} + \kappa \cdot \I[f'].
\end{equation*}
To conclude the proof, it suffices to observe that $\I[f']\leq M^2/(4(1-\kappa))$ implies $\I[f]\leq M^2/(4(1-\kappa))$ because
\begin{align*}
\I[f] &\leq \frac{M^2}{4} + \kappa \cdot \I[f'] \\
&\leq \frac{M^2}{4} + \kappa \cdot \frac{M^2}{4(1-\kappa)} \\
&= \frac{M^2}{4} \left(1 + \frac{\kappa}{1-\kappa}\right) \\
&\leq M^2/(4(1-\kappa)).
\end{align*}
\qed
\end{proof}

In the Boolean case, $\kappa=1/2$ and $M=1$, so that the upper bound $M^2/(4(1-\kappa))$ equals $1/2$ and the above result is a generalization of the result in \cite{Lau13}.

The above proof is also significantly simpler than the one in \cite{Lau13}. It can be easily checked that in the Boolean case, our argument on variance essentially amounts to compute the fraction of edges in the Hamming cube which are boundary edges.

\section{Concluding remarks}

\label{concl}

In the context of modelling biological systems, the choice of a relevant function with respect to the biological application is a critical and difficult step in the modelling process, notably because of the large (exponential in $n$) amount of functions compatible with the state transition graph. A current challenge is thus to find good selection criteria. We have already mentioned that canalizing functions are significantly predominant in gene network modelling \cite{Sub22}, and that networks with NC rules are stable \cite{Kau04}. While many works focus on the NC property for Boolean functions \cite{Kau04,Sub22}, we did not find any study of the canalizing properties for biological multivalued functions in the literature.

We have mentioned that in the process of modelling gene networks, canalization has been used as a guide to select candidate Boolean functions to parameterise the regulatory graph \cite{HJ12,Zho16}. With the extensions proposed in this paper, such an approach could be considered for multivalued functions as well. Moreover, the simple example of phage lambda modelling (Sections \ref{sec:phage} and \ref{sec:phage2}) suggests that it would be worth taking into account, in the sensitivity analysis, the updating rules and the notions of both NC and WNC functions.


On the theoretical side, an obvious question is whether the bound $M^2/(4(1-\kappa))$ in Theorem \ref{mainth} can be improved for multivalued WNC, or at least NC, functions, along the lines of \cite{Klo13}.

\bibliographystyle{plain}

\end{document}